\documentclass[A4,11pt]{article}

\usepackage[a4paper,margin=2.45cm]{geometry}

\usepackage{authblk}

\usepackage{amssymb,amsmath,amsthm}

\usepackage{mathtools,enumitem,mathabx}
\usepackage[linktoc=all,hidelinks]{hyperref}

\DeclareMathOperator{\supp}{supp}
\DeclareMathOperator{\diag}{diag}

\newcommand{\bra}[1]{\langle#1\rangle}
\newcommand{\norm}[1]{\|#1\|_{\ell^2}}
\newcommand{\eval}[2]{\left.#1\right|_{#2}}
\newcommand{\F}[1]{\widehat{\mathbf 1}_{#1}}

\newtheorem{Thm}{Theorem}
\newtheorem{Prop}{Proposition}
\newtheorem{Conj}{Conjecture}
\newtheorem{Ex}{Example}

\title{Number of cyclic $n$-roots and \\ disjointness of time-frequency supports}
\author{Weiqi Zhou\thanks{zwq@xzit.edu.cn}}
\affil{\small School of Mathematics and Statistics, Xuzhou University of Technology \\  {\footnotesize Lishui Road 2, Yunlong District, Xuzhou, Jiangsu Province, China 221018}}
\date{}							

\begin{document}
\maketitle
\begin{abstract}
A cyclic $n$-root is an $n$-dimensional complex vector that solves a particular set of multivariate polynomial equations. There is a one-to-one correspondence between unimodular cyclic $n$-roots and bi-unimodular vectors (CAZAC sequences) with leading entry one. It was conjectured by Bj{\"o}rck and Saffari that the set of cyclic $n$-roots is finite if and only if $n$ is square-free. It is known that such a set is infinite if $n$ is not square-free, and finite if $n$ is prime. A critical reduction in Haagerup's proof for prime $n$ is to show that infinity of cyclic $n$-roots (for any $n$) implies existence of two vectors with disjoint time-frequency supports. In this paper we show that such pair of vectors always exist if $n$ is composite, indicating that the original reduction is not adequate for composite square-free cases. A discussion on the existence of a single vector whose support is disjoint with its Fourier transform is also included. \\

{\noindent
\\[1ex]
}
\end{abstract}

\section{Introduction}
A vector $z=(z_0,z_1,\ldots,z_{n-1})^T\in\mathbb C^n$ is called a \emph{cyclic $n$-root} if it solves the following system of multivariate polynomial equations:
\begin{equation} \label{EqCycR}
\begin{dcases}
z_0+z_1+\ldots+z_{n-1}=0, \\
z_0z_1+z_1z_2+\ldots+z_{n-1}z_0=0, \\
\quad\quad\quad\quad\quad\quad\vdots \\
z_0z_1\cdots z_{n-2}+\ldots+z_{n-1}z_0\cdots z_{n-3}=0, \\
z_0\cdots z_{n-1}=1.
\end{dcases}
\end{equation}
The homogeneous products in the first $n-1$ equations have to be taken circulantly instead of arbitrarily, otherwise by the Vieta theorem, $z$ can only be a permutation of roots of unity. Cyclic $n$-roots are initially introduced in \cite{bjorck1990} to search for bi-unimodular vectors (CAZAC sequences, see Section \ref{SecRef} for details). As of now general complex solutions of \eqref{EqCR} has become a research area of independent interest. For small numbers $n=2,3,5,6,7,10,11,13,14$, with the aid of computers, their numbers are completely known, see e.g., tables in \cite{benedetto2019, fuhr2015}. It is also conjectured that

\begin{Conj}[\cite{bjorck1995}] \label{ConjBS}
The number of cyclic $n$-roots is finite if and only if $n$ is square-free.
\end{Conj}

Here a number is called \emph{square-free} if it is the product of distinct primes, so that it has no divisor other than $1$ that is a perfect square. The only if part of this conjecture is established in \cite{bjorck1995}. The if part is only known for prime $n$, and is proved in \cite{haagerup2008}.

Let $\zeta_n=e^{2\pi\mathbf i/n}$ and $F=n^{-1/2}[\zeta_n^{ij}]_{i,j=0}^{n-1}$ the Fourier matrix. Let $x=(x_0,\ldots,x_{n-1})^T\in\mathbb C^n$, denote by $\hat x=Fx$ its Fourier transform and $\check x=F^{-1}x$ its inverse Fourier transform. A critical step in the proof of prime cases in \cite{haagerup2008} is the following reduction:

\begin{Prop}[{\cite[Lemma 3.4]{haagerup2008}, \cite[Theorem 5.5]{benedetto2019}}] \label{PropCRR}
For a given $n$, if there are infinitely many cyclic $n$-roots, then there exist $u,v\in\mathbb C^n\setminus\{0\}$ such that
\begin{equation} \label{EqCRR}
u_iv_i=\hat u_i\check v_{i}=0
\end{equation}
holds for all $i\in\{0,1,\ldots,n-1\}$. 
\end{Prop}

Denote the support of $x$ by
$$\supp(x)=\{k: x_k\neq 0\}.$$
Then an equivalent formulation of \eqref{EqCRR} is 
\begin{equation} \label{EqCRR2}
\supp(u)\cap\supp(v)=\supp(\hat u)\cap\supp(\check v)=\emptyset.
\end{equation}

The main result of this paper is the following:
\begin{Thm} \label{ThmMain}
There exist $u,v\in\mathbb C^n\setminus\{0\}$ that satisfy \eqref{EqCRR} if and only if $n$ is composite. Moreover, for any divisor $d$ of $n$ that satisfies $1<d<\sqrt{2n}$, there exist $1+\lfloor d/2\rfloor$ number of vectors in $\mathbb C^n\setminus\{0\}$ that mutually satisfy \eqref{EqCRR}.
\end{Thm}

Theorem \ref{ThmMain} does not imply that the number of cyclic $n$-roots is infinite if $n$ is a square-free composite number, it merely means that the reduction step in Proposition \ref{PropCRR} is not adequate for proving finiteness of cyclic $n$-roots for a square-free composite number $n$. Since cyclic $n$-roots are completely known (\cite{bjorck1991}) for $n=6=2\times 3$, it is perhaps more convincing to look at an example of $n=6$ to see this:
\begin{Ex} \label{ExCR}
Let $n=6$ and set
$$u=(0,1,0,0,-1,0)^T, \quad v=(1,0,0,1,0,0)^T,$$
then $u,v$ satisfy \eqref{EqCRR}. In fact, $\hat u$ is supported on odd indices $\{1,3,5\}$, while $\check v=\hat v$ is supported on even indices $\{0,2,4\}$.
And $\supp(u)\cap\supp(v)=\emptyset$ is obvious, thus \eqref{EqCRR2} holds.
\end{Ex}

The rest of this paper is organized as follows: In the next section we present the Fourier analytical preliminaries of Example \ref{ExCR} and prove Theorem \ref{ThmMain}. We then give a brief discussion on the existence of a single vector whose support is disjoint with its Fourier transform and establish a partial result (Proposition \ref{PropZPD}). In the last section we provide backgrounds on cyclic $n$-roots, as well as its connections to bi-unimodular vectors, CAZAC sequences and circulant Hadamard matrices. To make the paper self-contained, an outline of Haagerup's proof of the prime cases is also included. 

\section{Vectors with disjoint time-frequency supports}
\subsection{The Fourier transform on finite cyclic groups}
Let $\mathbb Z_n=\{0,1,\ldots,n-1\}$ be the cyclic group of order $n$, the dual group of $\mathbb Z_n$ is $\widehat{\mathbb Z}_n=\{\chi_s(t)\}_{s=0}^{n-1}$ where the group characters are functions on $\mathbb Z_n$ defined as 
$$\chi_s(t)=e^{2\pi\mathbf i st/n}=\zeta_n^{st}, \quad t\in\mathbb Z_n.$$

A vector $x=(x_0,\ldots,x_{n-1})^T\in\mathbb C^n$ is at the same time the function $f: k\mapsto x_k$ in $\ell^2(\mathbb Z_n)$, the Fourier transform and its inverse of a function $f\in\ell^2(\mathbb Z_n)$ are still functions on $\mathbb Z_n$ and are defined respectively as
\begin{align*}
\hat f(t)&=\frac{1}{\sqrt n}\sum_{s\in\mathbb Z_n}f(s)\chi_s(t)=\frac{1}{\sqrt n}\sum_{s\in\mathbb Z_n}f(s)\zeta_n^{st}, \quad t\in\mathbb Z_n, \\
\check f(t)&=\frac{1}{\sqrt n}\sum_{s\in\mathbb Z_n}f(s)\chi_s(-t)=\frac{1}{\sqrt n}\sum_{s\in\mathbb Z_n}f(s)\zeta_n^{-st}, \quad t\in\mathbb Z_n.
\end{align*}

For any subset $A$ of $\mathbb Z_n$ we use the notation $\mathbf 1_A$ for the characteristic function (vector) on $A$ whose $k$-th term is defined as
$$\mathbf 1_A(k)=\begin{cases}1 & k\in A, \\ 0 & k\notin A.\end{cases}$$
In particular we have
$$\F{A}(t)=\frac{1}{\sqrt n}\sum_{a\in A}\chi_a(t)=\frac{1}{\sqrt n}\sum_{a\in A}\zeta_n^{at}.$$
One can verify that these are consistent with the definitions of $\hat x=Fx$ and $\check x=F^{-1}x$ using the Fourier matrix $F$. The Fourier matrix $F$ is just the character table $F=n^{-1/2}[\chi_i(j)]_{i,j=0}^{n-1}$.

Let 
$$T:(x_0,x_1,\ldots,x_{n-1})^T\mapsto(x_{n-1},x_0,\ldots,x_{n-2})^T$$ 
be the discrete translation operator, and 
$$M=\diag(1,\zeta_n, \zeta_n^2, \ldots,\zeta_n^{n-1})$$ 
the discrete modulation operator. $T^n=M^n=I$ is the identity operator. It is well known that $T$ is diagonalized by $F$ with $M$ being the diagonal matrix after diagonalization (the discrete convolution theorem), and 
\begin{equation} \label{EqTM}
FTx=M\hat x, \quad FMx=T^{-1}\hat x
\end{equation}
holds for all $x\in\mathbb C^n$. Let $d$ be a divisor $n$, and set $m=n/d$, denote by
$$H_d=m\mathbb Z_d=\{0,m,\ldots,(d-1)m\}$$
the subgroup of order $d$ in $\mathbb Z_n$, and 
\begin{equation} \label{EqPerp}
H_d^{\perp}=\{a\in\mathbb Z_n: ab \bmod n\equiv 0 \text{ holds for all } b\in H_d\}=d\mathbb Z_m=\{0,d,\ldots,(m-1)d\}
\end{equation}
the orthogonal subgroup of $H_d$. A core ingredient of the counterexamples is the following orthogonal relation (an analog of the Poisson summation formula), which can be verified by straightforward computation:
\begin{equation} \label{EqPoisson}
\F{H_d}=\frac{d}{\sqrt n}\cdot \mathbf 1_{H_d^{\perp}}. 
\end{equation}

\subsection{The case of several vectors}
\begin{proof}[\textbf{proof of Theorem \ref{ThmMain}}]
\begin{description}[leftmargin=*]
\item[The if part:] Suppose that $n$ is composite, $1<d<\sqrt{2n}$ is a divisor of $n$, and $m=n/d$. We will produce $1+\lfloor d/2\rfloor$ numbers of vectors $u^{(0)},  u^{(1)}, \ldots, u^{(\lfloor d/2\rfloor)}$ that mutually satisfy $\eqref{EqCRR2}$, and thus also $\eqref{EqCRR}$.

Let $u^{(0)}=\mathbf 1_{H_d}$, then 
\begin{equation} \label{EqSuppU2}
\supp(u^{(0)})=H_d=m\mathbb Z_d.
\end{equation}
Applying \eqref{EqPoisson} first, and then \eqref{EqPerp}, we get
$$\hat u^{(0)}=\F{H_d}=\frac{d}{\sqrt n}\cdot \mathbf 1_{H_d^{\perp}}=\frac{d}{\sqrt n}\cdot \mathbf 1_{H_m}.$$
Since $H_m=-H_m$ also holds, we get
\begin{equation} \label{EqSuppFU2}
\supp(\check u^{(0)})=\supp(\hat u^{(0)})=H_m=d\mathbb Z_m.
\end{equation}
For each $k\in\{1,\ldots,\lfloor d/2\rfloor\}$, we define $u^{(k)}$ to be
$$u^{(k)}=T^kM^{-k}u^{(0)}.$$
Then \eqref{EqTM} implies 
$$\hat u^{(k)}=M^kT^k\hat u^{(0)}, \quad \check u^{(k)}=M^{-k}T^{-k}\check u^{(0)}.$$
Since modulation does not change the support, and translation shifts the support, we obtain for each $k\in\{0,1,\ldots,\lfloor d/2\rfloor\}$ that
\begin{align*} 
\supp(u^{(k)})&=k+H_d=k+m\mathbb Z_d, \\
\supp(\hat u^{(k)})&=k+H_m=k+d\mathbb Z_m, \\
\supp(\check u^{(k)})&=-k+H_m=-k+d\mathbb Z_m. 
\end{align*}
Now $d<\sqrt{2n}$ implies $m=n/d>d/2\ge\lfloor d/2\rfloor$, hence $\supp(u^{(0)}), \supp(u^{(0)}), \ldots, \supp(u^{(\lfloor d/2\rfloor)})$ are mutually disjoint. And if $k,k'\in\{0,1,\ldots,\lfloor \frac{d}{2}\rfloor\}$ are distinct, then we have
$$d-k'\ge\lfloor\frac{d}{2}\rfloor\ge k,$$
where the equality can not be attained since it holds if and only if $d$ is even and $k=k'$. Therefore $\supp(\hat u^{(k)})=k+d\mathbb Z_m$ is disjoint from $\supp(\check u^{(k')})=-k'+d\mathbb Z_m=(d-k')+d\mathbb Z_m$. Together we conclude that any two vectors from $u^{(0)}, u^{(1)}, \ldots, u^{(\lfloor d/2\rfloor)}$ satisfy \eqref{EqCRR2}, and thus also \eqref{EqCRR}.

\item[The only if part:] The contrapositive direction that ``if $n$ is prime (not composite), then there exist no such pairs'' is exactly the final step in Haagerup's proof. We repeat it from \cite[Theorem 3.5]{haagerup2008} (see also \cite[Theorem 5.6]{benedetto2019}) here:

Suppose $n=p$ is a prime number. Assume, for the purpose of contradiction that such pair $u,v$ exist, then \eqref{EqCRR2} can be quantitive described as
\begin{equation} \label{EqUpper}
|\supp(u)|+|\supp(v)|+|\supp(\hat u)|+|\supp(\check v)|\le p+p=2p.
\end{equation}
However, the additive uncertainty principle \cite{tao2005} states that for any $x\in\mathbb C^p\setminus\{0\}$ we should have
\begin{equation} \label{EqUPp}
|\supp(x)|+|\supp(\hat x)|\ge p+1.
\end{equation} 
Applying \eqref{EqUPp} on $u$ and $v$ respectively we get
$$|\supp(u)|+|\supp(\hat u)|+|\supp(v)|+|\supp(\hat v)|\ge (p+1)+(p+1)=2p+2,$$
which contradicts \eqref{EqUpper}. 
\end{description}
\end{proof}

\subsection{The case of a single vector}
In this part we consider a closely related but slightly different problem: When do we have $x\in\mathbb C^n\setminus\{0\}$ that satisfy
\begin{equation} \label{EqCRR3}
\supp(x)\cap\supp(\hat x)=\emptyset.
\end{equation} 

Let
$$\eval{F}{S\times S'}=[F_{i,j}]_{i\in S, j\in S'}$$ 
be the submatrix of $F$ obtained by taking row indices from $S$ and column indices from $S'$. Similarly for a vector $x$, the notation 
$$\eval{x}{S}=[x_i]_{i\in S}$$
means the restriction of $x$ to $S$, which is a vector in $\mathbb C^{|S|}$ with index set $S$. If $S=S'$, then the submatrix $\eval{F}{S\times S'}$ is called a \emph{principal submatrix} of $F$, its determinant is called a \emph{principal minor} of $F$ (accordingly a \emph{minor} of $F$ is the determinant of a square submatrix that is not necessarily principal). If we denote $F'=\eval{F}{\supp(x)\times\supp(x)}$, then it is clear that
\begin{equation} \label{EqCRR4}
\exists x \text{ that satisfies \eqref{EqCRR3}} \quad \Leftrightarrow\quad \exists x \text{ s.t. } \eval{F'x}{\supp(x)}=0 \quad\Leftrightarrow\quad \det F'=0.
\end{equation}
Therefore the task of looking for $x$ that satisfies \eqref{EqCRR3} translates into identifying singular principal submatrices of $F$. Then $x$ can be obtained by embedding vectors from the kernel of such singular submatrices into the full space.

If $p$ is prime, then the Chebotar{\"e}v theorem (see e.g., \cite{stevenhagen1996}) asserts that all minors of the $p \times p$ Fourier matrix are non-vanishing, thus no such $x$ exists. The Chebotar{\"e}v theorem also implies the additive uncertainty principle \eqref{EqUPp}.

If $p^m$ is a prime power (i.e., $p$ is prime, $m\ge 2$ is a natural number), $F$ is the $p^m\times p^m$ Fourier matrix, and $H_p=p^{m-1}\mathbb Z_p$ the subgroup of order $p$ in $\mathbb Z_{p^m}$, then it is easy to see that $\eval{F}{H_p\times H_p}$ is the all one matrix and is thus singular.

In light of these observations, it is recently proposed in \cite{caragea2025} (see also \cite{cabrelli2025}) that the following extension of the Chebotar{\"e}v Theorem could be true:
\begin{Conj}[\cite{caragea2025}] \label{ConjCLMP}
All principal minors in the $n\times n$ Fourier matrix are non-vanishing for square-free $n$.
\end{Conj}
Partial results for various forms of square-free $n$ have been established in \cite{romanos2025,loukaki2025}.

We would like to mention that the scope of Conjecture \ref{ConjCLMP} can easily be, and considering the observation in Proposition \ref{PropZPD} also shall be, extended to all finite Abelian groups. If $G=\mathbb Z_{n_1}\times\ldots\times\mathbb Z_{n_d}$ is a finite Abelian group, then its dual is
$$\widehat G=\widehat{\mathbb Z}_{n_1}\times\ldots\times\widehat{\mathbb Z}_{n_d}=\{\chi_s(t)\}_{s\in G},$$
where the group characters are functions on $G$ defined as 
$$\chi_s(t)=\chi_{s_1}(t_1)\cdot \chi_{s_2}(t_2)\cdots\chi_{s_d}(t_d)$$ 
with $s=(s_1,\ldots, s_d)\in G, t=(t_1,\ldots, t_d)\in G$, and $\chi_{s_k}(t_k)\in\widehat{\mathbb Z}_{n_k}$ for each $k\in\{1,\ldots,d\}$. The Fourier transform of a function $f\in\ell^2(G)$ is still a function on $G$ and is defined as
$$\hat f(t)=\frac{1}{\sqrt{|G|}}\sum_{s\in G}f(s)\chi_s(t), \quad t\in G.$$
The Fourier matrix for $G$ is the character table $F=|G|^{-1/2}[\chi_s(t)]_{s,t\in G}$. 

If $G,G'$ are two finite Abelian groups, and $x\in\ell^2(G)$ satisfies \eqref{EqCRR3}, then the embedding $x\mapsto (x,0)$ easily produces a new instance in $G\times G'$ that satisfies \eqref{EqCRR3}. Therefore, to understand whether there exists $x$ that can satisfy \eqref{EqCRR3} on a finite Abelian group, it suffices to understand its existence in the prime power component $\mathbb Z_{p^m}$ with $m\ge 2$ (true), the square-free component $\mathbb Z_{p_1}\times\cdots\times\mathbb Z_{p_k}$ (conjectured not), and the prime product component $\mathbb Z_p^d$ (to be elaborated below). For $\mathbb Z_p^d$ we have:

\begin{Prop} \label{PropZPD}
Let $p$ be a prime number and $d\ge 2$ a natural number, then there exists some $x\in\ell^2(\mathbb Z_p^d)\setminus\{0\}$ that satisfy \eqref{EqCRR3}.
\end{Prop}

\begin{proof}
Although the last construction in this proof below actually works for all cases, we seek, with a bit extra effort, to produce $x$ of minimal support size.

If $p=2$, then we set $S=\{0,(1,1,\underbrace{0,\ldots,0}_{d-2 \text{ terms}})\}$. It is easy to see that $\eval{F}{S\times S}=\begin{pmatrix} 1 & 1 \\ 1 & 1 \end{pmatrix}$ is singular.

If $p=4k+1$ for some $k\in\mathbb N$, then $-1$ is a quadratic residue modulo $p$. The number $a=(2k)!$ is a square root of $-1$, since by Wilson's theorem we have 
$$a^2=\left((2k)!\right)^2=(-1)^{2k}((2k)!)^2=(2k!)\cdot\prod_{j=1}^{2k}(-j)=(4k)!=(p-1)!=-1 \pmod p.$$
Thus we get $1^2+a^2=0\pmod p$. Therefore setting $S=\{0,(1,a,\underbrace{0,\ldots,0}_{d-2 \text{ terms}})\}$, we again get that $\eval{F}{S\times S}=\begin{pmatrix} 1 & 1 \\ 1 & 1 \end{pmatrix}$ is singular. 

If $p=4k+3$ for some $k\in\mathbb N$ and $d\ge 3$, then first we notice that the equation 
\begin{equation} \label{EqAB}
1+a^2+b^2=0
\end{equation} 
admits a pair of solutions $a,b$ in $\mathbb Z_p$. Indeed, every pair of non-zero elements $\pm r\in\mathbb Z_p$ produces a unique quadratic residue $(\pm r)^2$. Thus taking $0$ also into consideration, in total there are $1+(p-1)/2=(p+1)/2$ quadratic residues. If $Q$ is the set of quadratic residues in $\mathbb Z_p$ and $Q'=-1-Q$, then we have $|Q|=|Q'|=(p+1)/2$. Therefore $Q$ and $Q'$ must intersect non-trivially. Then by writing an element in the intersection in two different way, as $a^2$ (since it belongs to $Q$) and as $-1-b^2$ (since it belongs to $Q'$), solves \eqref{EqAB}. Now we set $S=\{0,(1,a,b,\underbrace{0,\ldots,0}_{d-3 \text{ terms}})\}$, then again we get that $\eval{F}{S\times S}=\begin{pmatrix} 1 & 1 \\ 1 & 1 \end{pmatrix}$ is singular.

In all cases above, existence of such $x$ then follows from \eqref{EqCRR4}.

For the final case $\mathbb Z_p^2$ with $p=4k+3$ for some $k\in\mathbb N$, it is not possible to produce $2\times 2$ singular principal submatrices any more, the minimal support size for such $x$ seem to be difficult to determine either, but an explicit construction of such $x$, with support size $2p-2$, is the following: let $\bra{(a,b)}$ denote the cyclic subgroup generated by $(a,b)$, then we may set
$$x=\mathbf 1_{\bra{(1,0)}}-\mathbf 1_{\bra{(1,1)}}.$$
Since all proper subgroups must intersect trivially in $\mathbb Z_p^2$, we easily get
$$\supp(x)=\left(\bra{(1,0)}\cup\bra{(1,1)}\right)\setminus\{(0,0)\}.$$
Next one may verify (either through direction computation, or use \eqref{EqPerp} and \eqref{EqPoisson}, the orthogonal subgroup can be defined analogously in finite Abelian groups) that
$$\hat x=\mathbf 1_{\bra{(0,1)}}-\mathbf 1_{\bra{(1,-1)}},$$
and then
$$\supp(\hat x)=\left(\bra{(0,1)}\cup\bra{(1,-1)}\right)\setminus\{(0,0)\}$$
is disjoint with $\supp(x)$. 
\end{proof}

Proposition \ref{PropZPD} suggests that Conjecture \ref{ConjCLMP} can be reasonably extended to

\begin{Conj}
Let $G$ be a finite Abelian group, then there exists $x\in\ell^2(G)\setminus\{0\}$ that satisfy \eqref{EqCRR3} if and only if $|G|$ is not square-free.
\end{Conj}

\section{Supplements on backgrouds} \label{SecRef}
\subsection{Bi-unimodular vectors and CAZAC sequences} 
A vector $x\in\mathbb C^n$ is called \emph{unimodular} if $|x_0|=\ldots=|x_{n-1}|=1$. A unimodular vector $x$ is called \emph{bi-unimodular} if $\hat x$ is also unimodular. We say that $x$ has \emph{leading entry $1$} if $x_0=1$.

Let $\bra{\cdot,\cdot}$ be the inner product on $\mathbb C^n$. A sequence $x=(x_0,x_1,\ldots,x_{n-1})^T\in\mathbb C^n$ is called a \emph{CAZAC (constant amplitude zero auto correlation) sequence} if $|x_0|=\ldots=|x_{n-1}|=1$ (unit constant amplitude) and $\bra{Tx,x}=\ldots=\bra{T^{n-1}x,x}=0$ (zero auto correlation). 

It is easy to verify that the constant amplitude condition $|x_0|=\ldots=|x_{n-1}|$ is equivalent to $\bra{Mx,x}=\ldots=\bra{M^{n-1}x,x}=0$. Indeed, writing the diagonal of $M^j$ into a vector we get $(1,\zeta_n^j,\ldots,\zeta_n^{(n-1)j})^T$, which are columns of $\sqrt nF$. Therefore $(Mx,x)=\dots=\bra{M^{n-1}x,x}=0$ is equivalent to the vector $(|x_1|,\ldots,|x_n|)$ being orthogonal to all columns except the all one column in $F$, which means $|x_1|=\ldots=|\hat x_n|$. 

Applying the Fourier transform, The vanishing conditions $(Mx, x)=\dots=\{M^{n-1}x, x\}=(Tx, x)=\dots=\{T^{n-1}x, x\}=0$ transform into $(T\hat x,\hat x)=\dots=\{T^{n-1}\hat x,\hat x\}=(M\hat x,\hat x)=\dots=\{M^{n-1}\hat x,\hat x\}=0$, which implies that $\hat x$ has zero auto correlation and constant amplitude. Finally the Parseval identity $\norm{x}=\norm{\hat x}$ shows that $|\hat x_i|$ has to be $1$ for all $i\in\{0,1,\ldots,n-1\}$. Therefore
\begin{equation} \label{EqCB}
x \text{ is CAZAC} \;\Leftrightarrow\; \hat x \text{ is CAZAC} \;\Leftrightarrow\; x \text{ is bi-unimodular} \;\Leftrightarrow\; \hat x \text{ is bi-unimodular}.
\end{equation}

CAZAC sequences finds many applications in radar and cellular communications (see e.g, the introductory part of \cite{benedetto2019}), in the theoretical aspect it can be used to construct circulant Hadamard matrices. An $n\times n$ matrix is \emph{circulant} if its columns are $\{x,Tx,\ldots,T^{n-1}x\}$ for some $x\in\mathbb C^n$, and \emph{Hadamard} if it is orthogonal and all its entries are on the unit circle in $\mathbb C$. It is immediate to see that if $x$ is CAZAC, then $[x,Tx,\ldots,T^{n-1}x]$ forms a circulant Hadamard matrix and vice versa:
$$x \text{ is CAZAC} \;\Leftrightarrow\; [x,Tx,\ldots,T^{n-1}x] \text{ is circulant Hadamard}.$$

One example of a CAZAC sequence is the Gaussian sequence $\{\zeta_n^{ak^2+bk}\}_{k=0}^{n-1}$ where $n$ is odd, $a,b\in\mathbb Z$ are fixed and $a$ is co-prime to $n$. Such a sequence is also an eigenvector of the discrete time-frequency shift $M^{2a}T$. CAZAC sequences need not contain only $n$-th roots of unity: The first non-trivial example is perhaps the Bj{\"o}rck sequence $(1,1,1,e^{i\theta},1,e^{i\theta},e^{i\theta})$ at $n=7$ for $\theta=\arccos(-3/4)$. Bj{\"o}rck sequences can be defined for all prime lengths \cite[Section 4.1]{benedetto2019} and as can be seen from the $n=7$ example, they may contain numbers of form $e^{2\pi i\alpha}$ for irrational $\alpha$. Bj{\"o}rck sequences are derived by counting the number of quadratic residues that remain quadratic residues upon shifting, for which there are explicit formulas (see e.g., \cite{perron1952}). The search of CAZAC sequences has been a central topic in the related applications, it is known that if $n$ is not square-free, then there are infinitely many CAZAC sequences with leading entry one \cite{backelin1989,bjorck1995}.

\subsection{Cyclic $n$-roots and CAZAC sequences} 
It is easy to verify that the following map is a bijection between unimodular cyclic $n$-roots and CAZAC sequences with leading entry $1$:
\begin{equation} \label{EqR2C}
(z_0,z_1,\ldots,z_{n-2},z_{n-1})\mapsto \left(\frac{x_1}{x_0}, \frac{x_2}{x_1}, \ldots, \frac{x_{n-1}}{x_{n-2}}, \frac{x_0}{x_{n-1}}\right), \quad x_0=1.
\end{equation}
Indeed, by construction, the last equation in \eqref{EqCycR} holds automatically since 
$$z_0\cdots z_{n-1}=\frac{x_1\cdots x_{n-1}x_0}{x_0\cdots x_{n-1}}=1,$$
while expanding other equations we see that
\begin{align*}
z_0+z_1+\ldots+z_{n-1}=0 \quad&\Leftrightarrow\quad \bra{T^{-1}x,x}=0, \\
z_0z_1+z_1z_2+\ldots+z_{n-1}z_0=0 \quad&\Leftrightarrow\quad \bra{T^{-2}x,x}=0, \\
&\;\;\vdots \\
z_0z_1\cdots z_{n-2}+\ldots+z_{n-1}z_0\cdots z_{n-3}=0 \quad&\Leftrightarrow\quad \bra{T^{-(n-1)}x,x}=0, \\
\end{align*}
thus the first $n-1$ equations in \eqref{EqCycR} corresponds to the vanishing auto correlation conditions one by one. Therefore

\begin{equation} \label{EqCR}
x \text{ is CAZAC with $x_0=1$} \;\Leftrightarrow\; z \text{ is a unimodular cyclic $n$-root}.
\end{equation}
Not all cyclic $n$-roots are unimodular, for example, among the $156$ cyclic $6$-roots found in \cite{bjorck1991}, only $48$ of them are unimodular \cite{haagerup1997}.

\subsection{Haagerup's reduction}
In order to make the paper self-contained, and understand the origin of Proposition \ref{PropCRR}, we briefly outline Haagerup's proof on finiteness of cyclic $n$-roots for prime $n$ below. The proof steps are repeated and condensed from the exposition in \cite{benedetto2019}. 

Let $z,x\in(\mathbb C\setminus\{0\})^n$ be as defined in \eqref{EqR2C} but without assuming $z$ being unimodular or $x$ being CAZAC, then \eqref{EqCR} transforms into 
\begin{align*}
z_0+z_1+\ldots+z_{n-1}=0 \quad&\Leftrightarrow\quad \frac{x_1}{x_0}+\ldots+\frac{x_{n-1}}{x_{n-2}}+\frac{x_0}{x_{n-1}}=0, \\
z_0z_1+z_1z_2+\ldots+z_{n-1}z_0=0 \quad&\Leftrightarrow\quad \frac{x_2}{x_0}+\ldots+\frac{x_0}{x_{n-2}}+\frac{x_1}{x_{n-1}}=0, \\
&\;\;\vdots \\
z_0z_1\cdots z_{n-2}+\ldots+z_{n-1}z_0\cdots z_{n-3}=0 \quad&\Leftrightarrow\quad \frac{x_{n-1}}{x_0}+\ldots+\frac{x_{n-3}}{x_{n-2}}+\frac{x_{n-2}}{x_{n-1}}=0, \\
\end{align*}
thus if we set $y=(1/x_0,\ldots,1/x_{n-1})$, then we get
\begin{equation} \label{EqCR2}
z \text{ is a cyclic $n$-root} \quad\Leftrightarrow\quad \begin{dcases}x_0=y_0=1, \\ x_iy_i=1, & i\in\{1,\ldots,n-1\}, \\ \sum_{i=0}^{n-1}x_{i+k}y_i=0, & k\in\{1,\ldots,n-1\}.\end{dcases}
\end{equation}
The last summation can be written concisely as $\bra{T^kx,\bar y}$ where $\bar y$ is the vector obtained by taking complex conjugates on each entry of $y$. Taking Fourier transform and setting $c_i=\hat x_i\hat y_{-i}$, we see that, same as the derivation of \eqref{EqCB}, the vector $(c_i)_{i=0}^{n-1}$ is in the span of $(1,1,\ldots,1)$. On the other hand, we also have
$$c_0+\ldots+c_{n-1}=\sum_{i=0}^{n-1}\hat x_i\hat y_{-i}=\sum_{i=0}^{n-1}\bra{\hat x,\bar{\check y}}=\bra{\hat x,\widehat{\bar y}}=\bra{x,\bar y}=\sum_{i=0}^{n-1}x_iy_i=n.$$
Hence each $c_i$ is $1$. Noticing also that $\hat y_{-i}=\check y_i$ holds, one can thus further reduce \eqref{EqCR2} to
\begin{equation} \label{EqCR3} 
z \text{ is a cyclic $n$-root} \quad\Leftrightarrow\quad \begin{dcases}x_0=y_0=1, \\ x_iy_i=1, & i\in\{1,\ldots,n-1\}, \\ \hat x_i\check y_{i}=1, & i\in\{1,\ldots,n-1\}.\end{dcases}
\end{equation}

If the set of solution pairs $(x,y)$ to the system of polynomial in \eqref{EqCR3} is infinite, then as an algebraic variety in $\mathbb C^{2n}$ it can not be compact, and must have a positive dimension and be unbounded. Then one can extract a sequence of pairs $(x^{(m)}, y^{(m)})$ from this algebraic variety such that 
$$\norm{x^{(m)}}^{2}+\norm{y^{(m)}}^{2}\to\infty.$$
Set 
$$c_m=\sum_{i=1}^{n-1}|x_i^{(m)}|^2=\norm{x^{(m)}}^{2}-|x_0^{(m)}|^2, \quad d_m=\sum_{i=1}^{n-1}|y_i^{(m)}|^2=\norm{y^{(m)}}^{2}-|y_0^{(m)}|^2,$$
and recall that $x_0^{(m)}=y_0^{(m)}=1$, one also gets
\begin{equation} \label{EqXYInf}
\norm{x^{(m)}}^{2}\cdot\norm{y^{(m)}}^{2}=(1+c_m)(1+d_m)\ge c_m+d_m+1=\norm{x^{(m)}}^{2}+\norm{y^{(m)}}^{2}-1\to\infty.
\end{equation} 
Now let
$$u^{(m)}=\frac{x^{(m)}}{\norm{x^{(m)}}}, \quad v^{(m)}=\frac{y^{(m)}}{\norm{y^{(m)}}},$$
be normalizations of $x^{(m)},y^{(m)}$ respectively. The normalized sequence $\{(u^{(m)},v^{(m)})\}_{m\in\mathbb N}$ is in a compact set, and admits a convergent subsequence $(u^{(m')},v^{(m')})\to(u,v)$. Then we must have $u\neq 0$ and $v\neq 0$ since $\norm{u^{(m')}}=\norm{v^{(m')}}=1$ holds for each $m'$, while \eqref{EqXYInf} indicates that
$$u_iv_i=\lim_{m'\to\infty} u_i^{(m')}\cdot v_i^{(m')}=\lim_{m'\to\infty} \frac{x_i^{(m')}\cdot y_i^{(m')}}{\norm{x^{(m')}}\cdot \norm{v^{(m')}}}=\lim_{m'\to\infty} \frac{1}{\norm{x^{(m')}}\cdot \norm{v^{(m')}}}=0,$$
$$\hat u_i\check v_{i}=\lim_{m'\to\infty} \hat u_i^{(m')}\cdot \check v_{i}^{(m')}=\lim_{m'\to\infty} \frac{\hat x_i^{(m')}\cdot \check y_{i}^{(m')}}{\norm{x^{(m')}}\cdot \norm{v^{(m')}}}=\lim_{m'\to\infty} \frac{1}{\norm{\hat x^{(m')}}\cdot \norm{\hat v^{(m')}}}=0,$$

Together these lead to Proposition \ref{PropCRR}, which is the core reduction for the proof that cyclic $p$-roots ($p$ is prime) are finite. The remaining steps are already repeated in the proof of the only if part of Theorem \ref{ThmMain}.

\end{document}